\documentclass{amsart}
\usepackage[numbers]{natbib}
\usepackage{xcolor,caption,subcaption,graphicx}
\usepackage{epstopdf}
\usepackage{placeins}
\usepackage{amsmath,amssymb,amsthm}
\usepackage{hyperref}
\usepackage[foot]{amsaddr}

\newtheorem{remark}{Remark}
\newtheorem{theorem}{Theorem}

\usepackage{amsmath}
\usepackage{amsfonts}
\usepackage{bm}

\newtheorem{assumption}{Assumption}
\newtheorem{proposition}{Proposition}

\newcommand{\vx}{\bm{x}}
\newcommand{\vy}{\bm{y}}
\newcommand{\vz}{\bm{z}}
\newcommand{\ve}{\bm{e}}

\newcommand{\vv}{\bm{v}}
\newcommand{\vu}{\bm{u}}

\newcommand{\vd}{\bm{d}}

\newcommand{\vp}{\bm{p}}
\newcommand{\vg}{\bm{g}}

\newcommand{\veta}{\bm{\eta}}

\newcommand{\sA}{\mathcal{A}}
\newcommand{\sR}{\mathcal{R}}
\newcommand{\sB}{\mathcal{B}}
\newcommand{\sW}{\mathcal{W}}

\newcommand{\sK}{\mathcal{K}}

\newcommand{\mM}{\mathsf{M}}

\newcommand{\mI}{\mathsf{I}}

\newcommand{\tr}{\mathsf{T}}
\newcommand{\diag}{\mathsf{diag}}

\newcommand{\KM}{\Gamma_{\text{KM}}}
\newcommand{\real}{\mathbb{R}}

\newcommand{\complex}{\mathbb{C}}

\newcommand{\bigO}{\mathcal{O}}
\newcommand{\CE}{\Psi}
\newcommand{\C}{\psi}
\renewcommand{\Re}{\text{Re}}
\renewcommand{\Im}{\text{Im}}
\newcommand{\ii}{\imath}

\DeclareMathOperator\nullspace{null}
\DeclareMathOperator\range{range}
\DeclareMathOperator\supp{supp}

\DeclareMathOperator\cond{cond}

\newcommand{\wh}{\widehat}
\newcommand{\ov}{\overline}
\newcommand{\wt}{\widetilde}

\title{Kirchhoff migration without phases}
\author{Patrick Bardsley}
\author{Fernando Guevara Vasquez}
\address{Mathematics Department, University of Utah, 155 S 1400 E RM
233, Salt Lake City UT 84112-0090.}
\email{bardsley@math.utah.edu, fguevara@math.utah.edu}

\begin{document}
\maketitle

%%%%%%%%%%%%%%%%%%%%%%%%%%%%%%%%%%%%%%%%%%%%%%%%%%%%%%%%%%%%%%%%%%%%%%%%
%% Abstract
\begin{abstract}
We present a simple, frequency domain, preprocessing step to Kirchhoff
migration that allows the method to image scatterers when the wave field
phase information is lost at the receivers, and only intensities are
measured.  The resulting imaging method does not require knowing the
phases of the probing field or manipulating the phase of the wave field
at the receivers. In a regime where the scattered field is small
compared to the probing field, the problem of recovering the
full-waveform scattered field from intensity data can be formulated as
an embarrassingly simple least-squares problem.  Although this only
recovers the projection (on a known subspace) of the full-waveform
scattered field, we show that, for high frequencies, this projection
gives Kirchhoff images asymptotically identical to the images obtained
with full waveform data. Our method can also be used when the source is
modulated by a Gaussian process and autocorrelations are measured at an
array of receivers.  \end{abstract}

%%%%%%%%%%%%%%%%%%%%%%%%%%%%%%%%%%%%%%%%%%%%%%%%%%%%%%%%%%%%%%%%%%%%%%%%
%% Classification numbers and keywords
% 78A46: optics, inverse scattering
% 35R30: PDE, inverse problems
\noindent
\\

%%%%%%%%%%%%%%%%%%%%%%%%%%%%%%%%%%%%%%%%%%%%%%%%%%%%%%%%%%%%%%%%%%%%%%%%
%% Intro
\section{Introduction}\label{sec:intro}
% What's the problem?
Imaging scatterers in a homogeneous medium from full waveform
data is well understood. The medium is probed with waves emanating
from one or more sources and the reflections from scatterers in the
medium are recorded at one or more receivers. An image of the scatterers
can be formed from these recordings by using classic imaging methods
such as Kirchhoff or travel time migration (see e.g.\@
\cite{Bleistein:2001:MMI}); or MUSIC (see e.g.\@ \cite{Cheney:2001:LSM}).
Both Kirchhoff migration and MUSIC rely on full-waveform measurements at
the receivers to form an image. Here we work in the frequency domain and
we assume that phase information is lost and only \emph{intensities} can
be measured at the receivers. To be more precise, if
$\wh{u}_r(\omega)$ is the wave field at receiver $r$ and 
angular frequency $\omega$, we can only measure the intensity
$|\wh{u}_r(\omega)|^2$. Intensity measurements arise in a variety
of physical problems, e.g. when the response time of a
receiver is much larger than a typical wave period. Such is the case in
optical coherence tomography \cite{Schmitt:1997:OCM,Schmitt:1999:OCT}
and diffraction tomography \cite{Gbur:2002:DTW, Gbur:2004:ICS}. In these
situations, intensities are much easier and more cost-effective to
measure than full waveform data.

The setup we analyze consists of one source and $N$ receivers, all of
known location. The receivers can only record intensities and only the
source intensity is known. If the scattered field is small compared to
the probing field (at the receivers) then the scattered field projected
onto a known subspace can be found from the intensity data by solving an
underdetermined real least squares problem of size $N \times 2N$ (per
frequency sample). This system is underdetermined because we are trying
to use the intensity data, i.e.  $N$ real measurements, to recover the
scattered field, i.e. $N$ complex or $2N$ real numbers. Fortunately, a
stationary phase argument shows that the error made by projecting the
scattered field does not affect Kirchhoff imaging (for high
frequencies).  Moreover the least-squares problem is typically
well-conditioned and its solution is embarrassingly simple: it merely
costs about $2N$ complex operations (additions or multiplications). Hence
our method is comparable in computational cost to Kirchhoff migration.
Well-known resolution studies for Kirchhoff migration can also be used for our method.

%%%%%%%%%%%%%%%%%%%%%%%%%%%%%%%%%%%%%%%%%%%%%%%%%%%%%%%%%%%%%%%%%%%%%%%%
\subsection{Related work for imaging with intensities}
% What have others done? Intensities
% Phase retrieval

One way of imaging with intensities is called ``phase retrieval'' and
consists of first recovering the phases from the intensity data and then
using the reconstructed field to image. Examples of this approach
include using intensity measurements at two different planes to recover
full waveform measurements at a single plane \cite{Gbur:2002:DTW},
iterative approaches e.g. \cite{Gerchberg:1972:GSA, Fienup:1982:PRA,
Maleki:1993:PRI, Crocco:2004:ISPMCC} and using differential identities
to relate intensity measurements with full waveform data (e.g.
\cite{Teague:1983:DPR, Holmes:2004:CRE}).  Other methods treat
intensity-only measurements as noisy measurements of full waveform data
(e.g. \cite{Devaney:1989:SDIM}) or use optimization techniques to fit
assumed models of scatterers to measured intensity data (e.g.
\cite{Takenaka:1997:RAIM}). 

The problem of imaging a few point scatterers can be reformulated as a
convex optimization problem involving low rank matrices
\cite{Chai:2011:AII, Candes:2013:PL, Xin:2015:PLO}.  Alternatively the
polarization identity $4\Re(\vu^*\vv)=\|\vu+\vv\|^2-\|\vu-\vv\|^2,
\vu,\vv\in\complex^N$, and linear combinations of single source
experiments can be used to recover dot products of two single source
experiments \cite{Novikov:2014:ISI}. Notice that recovering the dot
product $\vu^* \vv$ from the polarization identity requires multiplying
$\vu$ and $\vv$ by $\pm 1$ or $\pm \imath$. This requires manipulation of
the source phases, by e.g. introducing delays.  MUSIC can then be used
to image with this quadratic functional of the full waveform data
\cite{Novikov:2014:ISI}. Instead of directly controlling the source
phases, in \cite{Bardsley:2015:PCSP} we use two sources that send
exactly the same signal from two different locations. The problem of
recovering the full-waveform scattered field as measured at one receiver
location is formulated as a least-squares problem, which is analyzed in
\cite{Bardsley:2015:PCSP}.  The least-squares systems are typically $2N
\times 2N$ and the scattered field can be recovered up to a one
dimensional nullspace that does not affect the Kirchhoff images.  In
contrast, the present method requires less measurements (only $N$), no
pairwise illuminations (remark~\ref{rem:previous}) and the least-squares
systems are trivial to solve and (usually) well-conditioned.

%%%%%%%%%%%%%%%%%%%%%%%%%%%%%%%%%%%%%%%%%%%%%%%%%%%%%%%%%%%%%%%%%%%%%%%%
\subsection{Imaging with correlations}
% Why do we care?
Correlations are used for imaging when the sources are not well known.
For example in seismic imaging the sources may have unknown locations
\cite{Schuster:1996:RLC,Schuster:2004:ISI,Schuster:2009:SI} and may even
consist of ambient noise
\cite{Garnier:2009:PSI,Garnier:2010:RAI,Garnier:2015:SNR}. Fortunately,
the correlation of recordings at two points contains information about
the Green function of the medium between the two points
\cite{Garnier:2009:PSI}, which can be used to form an image of the
medium.  Correlations are also used in radar imaging
\cite{Cheney:2009:FRI}, since the operating frequencies make it
impractical to measure the phases at the receivers. In fact even
stochastic signals can be used in place of deterministic signals for the
probing fields \cite{Tarchi:2010:SARNR,Vela:2012:NRT}. 

As in \cite{Bardsley:2015:PCSP} we observe that autocorrelations (i.e.
correlating the recorded signal with itself) are equivalent to intensity
measurements (by the Wiener-Khinchin theorem, see e.g.
\cite{Ishimaru:1997:WPS}). Therefore the method we present here can be
applied to the case where only the autocorrelation (or the power
spectrum) of the source is known and autocorrelations (or power spectra)
are measured at the receivers. Since correlations are robust to additive
noise, we expect our method to work in low signal to noise ratio
situations, as we illustrate with numerical experiments.

% Setup
%%%%%%%%%%%%%%%%%%%%%%%%%%%%%%%%%%%%%%%%%%%%%%%%%%%%%%%%%%%%%%%%%%%%%%%%
\subsection{Contents}
The physical setup and notations we use are described in
\S\ref{sec:setup} and we briefly review Kirchhoff migration in
\S\ref{sec:kirchhoffrev}. Using the Born approximation, we formulate the
problem of recovering the full wave scattered field at the array from
intensity-only measurements as a linear least squares problem
\S\ref{sec:ionly}. In \S\ref{sec:phaseret} we analyze and solve the
least squares problem and show that its solution can be used with
Kirchhoff migration. We extend this imaging method to stochastic
illuminations and autocorrelation measurements in \S\ref{sec:stocillum}.
Numerical experiments for an optic regime are provided in
\S\ref{sec:numerics} and we conclude with a discussion in
\S\ref{sec:discussion}.

%%%%%%%%%%%%%%%%%%%%%%%%%%%%%%%%%%%%%%%%%%%%%%%%%%%%%%%%%%%%%%%%%%%%%%%%
\section{Wave propagation and intensity-only
measurements}\label{sec:waves} Here we introduce the setup we work with
and briefly recall Kirchhoff migration. Hereinafter we use the Fourier
transform convention for functions of time $t$:
\begin{equation}\label{eq:ft}
\wh{f}(\omega) = \int_{-\infty}^\infty dt f(t)e^{\ii\omega
t}, ~ f(t) = \frac{1}{2\pi}\int_{-\infty}^\infty
d\omega\wh{f}(\omega)e^{-\ii\omega t},
~~\mbox{for}~~
f(t),~\wh{f}(\omega)\in L^2(\real).
\end{equation}
%%%%%%%%%%%%%%%%%%%%%%%%%%%%%%%%%%%%%%%%%%%%%%%%%%
%% Wave propagation
\subsection{Wave propagation in a homogeneous medium}\label{sec:setup}
The physical setup we consider is depicted in figure~\ref{fig:setup}. We
probe the medium with a point source located at $\vec{\vx}_s$. Waves are
recorded on an array of receivers $\vec{\vx}_r = (\vx_r,0)\in\sA$ for
$r=1,\ldots, N$, where $\sA\subset\real^d\times\{0\}$ and $d=2,3$ is the
dimension. We use the notation $\vx$ for the first $d-1$ components of a vector $\vec{\vx} \in \real^d$. For simplicity we consider a linear array in 2D or a square
array in 3D, i.e.  $\sA=[-a/2,a/2]^{d-1}\times\{0\}$, however other
shapes may be considered. We impose only mild conditions on the
positions of the source and receivers, in particular that the source
is not in the array. We assume the medium contains scatterers with
reflectivity $\rho(\vec{\vy})$ with $\supp(\rho)=\sR$, and background
wave velocity $c_0$. 

%%%%%%%%%%%%%%%%%%%%%%%%%%%%%%%%%%%%%%%%%%%%%%%%%%
%% FIGURE - SETUP
\begin{figure}
\centering
\includegraphics[width=0.5\textwidth]{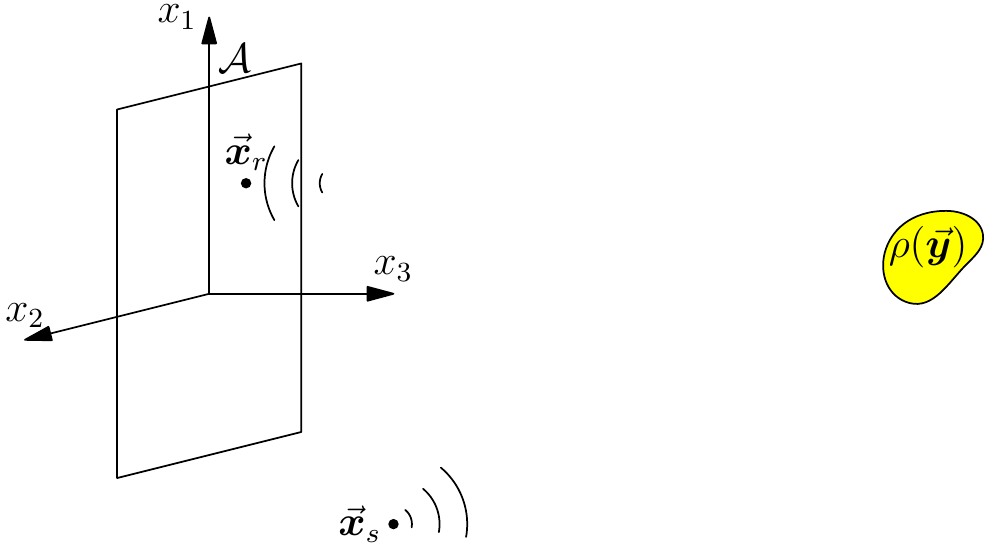}
\caption{Physical setup for an array $\sA$ of receivers with a wave
source located at $\vec{\vx}_s$. The scatterer is characterized by the
compactly supported function $\rho(\vec{\vy})$ (in yellow).}\label{fig:setup}
\end{figure}
%%%%%%%%%%%%%%%%%%%%%%%%%%%%%%%%%%%%%%%%%%%%%%%%%%

The total field arriving at the receiver location $\vec{\vx}_r$ from
frequency modulation $\wh{f}(\omega)$ at the source location
$\vec{\vx}_s$ is
\begin{equation}\label{eq:totalfield}
\wh{u}(\vec{\vx}_r,\vec{\vx}_s,\omega) =
\wh{G}(\vec{\vx}_r,\vec{\vx}_s,\omega)\wh{f}(\omega),
\end{equation}
where $\wh{G}$ is the Green's function for the (inhomogeneous) medium in the
frequency domain. We assume the scatterers are weak so that multiple
scatterings may be neglected and by the Born approximation
\begin{equation}\label{eq:bornappx}
\wh{G}(\vec{\vx}_r,\vec{\vx}_s,\omega) \approx
\wh{G}_0(\vec{\vx}_r,\vec{\vx}_s,\omega)+k^2\int_{\sR}d\vec{\vy}\rho(\vec{\vy})\wh{G}_0(\vec{\vx}_r,\vec{\vy},\omega)\wh{G}_0(\vec{\vy},\vec{\vx}_s,\omega),
\end{equation}
where $\wh{G}_0$ is the Green's function for the Helmholtz equation:
\begin{equation}\label{eq:greenf}
\wh{G}_0(\vec{\vx},\vec{\vy},\omega) =
\left\{\renewcommand{\arraystretch}{2.5}\begin{array}{ll}
\displaystyle\frac{\ii}{4}H_0^{(1)}(k|\vec{\vx}-\vec{\vy}|), & d=2,\\
\displaystyle\frac{\exp(\ii k|\vec{\vx}-\vec{\vy}|)}{4\pi|\vec{\vx}-\vec{\vy}|}, &
d=3.
\end{array}\right.
\end{equation}
Here $k=\omega/c_0$ is the wave number.

We express the total fields received on the array with linear algebra
notation as
\[
u(\vec{\vx}_r,\vec{\vx}_s,\omega) 
=
\ve_r^\tr\big(\vg_0(\vec{\vx}_s,\omega)+\vp(\vec{\vx}_s,\omega)\big)\wh{f}(\omega),\quad\text{for
$r=1,\ldots,N$,}
\]
where the vector $\vg_0$ is the vector of direct arrivals (or incident
field) at the array
\begin{equation}\label{eq:incvec}
\vg_0(\vec{\vx}_s,\omega) = \begin{bmatrix}
\wh{G}_0(\vec{\vx}_1,\vec{\vx}_s,\omega),
				\wh{G}_0(\vec{\vx}_2,\vec{\vx}_s,\omega),
				\cdots,
				\wh{G}_0(\vec{\vx}_N,\vec{\vx}_s,\omega)
				\end{bmatrix}^\tr,
\end{equation}
and the array response vector (or scattered field at the array) is
\begin{equation}\label{eq:refvec}
\vp(\vec{\vx}_s,\omega) = k^2\int_\sR d\vec{\vz}
\vg_0(\vec{\vz},\omega)\wh{G}_0(\vec{\vx}_s,\vec{\vz},\omega)\rho(\vec{\vz}).
\end{equation}

%%%%%%%%%%%%%%%%%%%%%%%%%%%%%%%%%%%%%%%%%%%%%%%%%%
%% Kirchhoff
\subsection{Kirchhoff migration}\label{sec:kirchhoffrev}
When full waveform measurements are available, i.e.
$u(\vec{\vx}_r,\vec{\vx}_s,\omega)$ is known for $r=1,\ldots,N$, the
scattered field $\vp$ can be obtained from the total field at the array,
$\vg_0$ and $\wh{f}(\omega)$. The scatterers in the medium can be imaged 
with Kirchhoff migration applied to $\vp$,
which for a single frequency $\omega$ has the form:
\begin{equation}\label{eq:kirchhoff} \KM\big[\vp,\omega\big](\vec{\vy})
=
\ov{\wh{G}}_0(\vec{\vx}_s,\vec{\vy},\omega)\vg_0(\vec{\vy},\omega)^*\vp(\vec{\vx}_s,\omega).
\end{equation} Here $\vec{\vy}$ is a point in the image. The Kirchhoff
migration functional has been studied extensively (see e.g.
\citet{Bleistein:2001:MMI} for a review). In the cross-range (the
direction parallel to the array), we can expect a resolution of $\lambda
L/a$ where $\lambda=2\pi/k$ is the wavelength, $L$ is the array to
scatterer distance and $a$ is array aperture. This is the Rayleigh
resolution limit. To obtain resolution in the range direction (the
direction perpendicular to the array), $\KM[\vp,\omega]$ needs to be
integrated over a frequency band $\sB$, e.g. $\sB =
[-\omega_\text{max},-\omega_\text{min}]\cup[\omega_\text{min},\omega_\text{max}]$.
In this case we can expect the resolution to be
$c_0/(\omega_\text{max}-\omega_\text{min})$.

%%%%%%%%%%%%%%%%%%%%%%%%%%%%%%%%%%%%%%%%%%%%%%%%%%
%% Intensity-only measurements
\subsection{Intensity-only measurements}\label{sec:ionly}
Using the illumination $\wh{f}(\omega)$ at the source location
$\vec{\vx}_s$, the \emph{intensity-only} measurement of the wave
field at $\vec{\vx}_r\in\sA$ is:
\begin{equation}
|u(\vec{\vx}_r,\vec{\vx}_s,\omega)|^2 =
|\wh{f}(\omega)|^2\ve_r^\tr\left[\ov{(\vg_0+\vp)} \odot
(\vg_0+\vp)\right],
\label{eq:intensity}
\end{equation}
where the operator $\odot$ denotes the componentwise or Hadamard product of two
vectors and $\{\ve_r\}_{r=1}^N$ is 
the standard orthonormal basis of $\real^N$. Our objective is to find as
much as we can about $\vp$ from the vector of measurements
$[|u(\vec{\vx}_r,\vec{\vx}_s,\omega)|^2]_{r=1,\ldots,N}$. This is done
by linearization, so we need to assume that the scattered field $\vp$
is small compared to the direct arrival $\vg_0$ at the array.
\begin{assumption}
\label{assump:smallness}
The position of the receivers, the source and the reflectivity are such
that $|\ve_r^\tr \vp| \ll |\ve_r^\tr \vg_0|$, $r=1,\ldots,N$.
\end{assumption}
This assumption is satisfied e.g. if the reflectivity is sufficiently small
and the source $\vec{\vx}_s$ is near the receiver array (as is shown
in figure~\ref{fig:setup}). With assumption~\ref{assump:smallness} we
can neglect quadratic terms in $\vp$ to approximate the intensity
measurements \eqref{eq:intensity} by a vector $\vd(\vec{\vx}_s,\omega)$ defined by
\[
|u(\vec{\vx}_r,\vec{\vx}_s,\omega)|^2 \approx \ve_r^\tr
\vd(\vec{\vx}_s,\omega) \equiv |\wh{f}(\omega)|^2\ve_r^\tr\Re\left[\ov{\vg}_0
\odot (\vg_0+2\vp)\right].
\]
This is not, strictly speaking, a linear system for $\vp \in \complex^N$
since $\vu \to \Re(\vu)$ is not a linear mapping from $\complex^N$ to
$\complex^N$.  However we can write an underdetermined linear system for
the real and imaginary parts of $\vg_0+\vp$ as follows
\begin{equation}\label{eq:linsys}
|\wh{f}(\omega)|^2\mM(\vec{\vx}_s,\omega)\begin{bmatrix}\Re(\vg_0+2\vp)\\\Im(\vg_0+2\vp)\end{bmatrix}
= \vd(\vec{\vx}_s,\omega),
\end{equation}
where the matrix $\mM(\vec{\vx}_s,\omega)\in\real^{N\times 2N}$ is given by 
\begin{equation}\label{eq:measmx}
\mM(\vec{\vx}_s,\omega) = \begin{bmatrix} \diag\big(\Re(\vg_0)\big) &
\diag\big(\Im(\vg_0)\big) \end{bmatrix}.
\end{equation}
We give in the next section an explicit solution to the least squares
problem \eqref{eq:linsys}.

%%%%%%%%%%%%%%%%%%%%%%%%%%%%%%%%%%%%%%%%%%%%%%%%%%%%%%%%%%%%%%%%%%%%%%%%
%% Nullspace and Kirchhoff migration 
\section{Migrating a least-squares estimate of the scattered field}\label{sec:phaseret} 
The first step in our imaging method consists of a cheap
least-squares preprocessing step  that gives an
approximation to the array response vector (\S\ref{sec:nullspace}). The second step is to
migrate this approximation with 
standard Kirchhoff migration \S\ref{sec:kirchhoff}. Crucially we show in
theorem~\ref{thm:km} that the mistake we make by using this
approximation of the array response vector does not affect the Kirchhoff images.

%%%%%%%%%%%%%%%%%%%%%%%%%%%%%%%%%%%%%%%%%%%%%%%%%%
%% Nullspace
\subsection{Recovering a projection of the array response vector}\label{sec:nullspace}

We start by finding a simple and explicit expression to the
pseudoinverse of the matrix $\mM$ that we obtained from linearizing the
problem of finding the real and imaginary parts of the array response
vector $\vp$. This can be used to recover from the data $\vd$ the
orthogonal projection of $[\Re(\vp)^\tr,\Im(\vp)^\tr]^\tr$ onto a known
$N$ dimensional subspace (that depends only on $\vg_0$). Moreover the
process is well conditioned.

First notice that the matrix $\mM$ is full-rank. Indeed a simple
calculation gives that $\mM \mM^\tr = \diag(\ov{\vg}_0 \odot \vg_0)$. This
matrix is clearly invertible because it is a diagonal matrix with the moduli
of 2D or 3D Green functions on the diagonal.  Hence the
Moore-Penrose pseudoinverse $\mM^\dagger$ can be written explicitly 
\begin{equation}\label{eq:mdagger}
\mM^\dagger = \mM^\tr(\mM\mM^\tr)^{-1}  =
\begin{bmatrix}\diag\big(\Re(\vg_0)\big)\\\diag\big(\Im(\vg_0)\big)\end{bmatrix}\diag(\ov{\vg}_0\odot\vg_0)^{-1}.
\end{equation}
We can use $\mM^\dagger$ to see what information about $\vp$ we can
recover from the right hand side $\vd$ in the least-squares problem
\eqref{eq:linsys}. Since $\mM$ has an $N$ dimensional nullspace, we can
only expect to recover the orthogonal projection of $[\Re(\vp)^\tr,
\Im(\vp)^\tr]^\tr$ onto $\range(\mM^\tr) = \left(\nullspace(\mM)\right)^\perp$. This
projection has a simple form when we write it in $\complex^N$, as can be
seen in the next proposition.
%%%%%%%%%%%%%%%%%%%%%%%%%%%%%%
%% Proposition (INVERSION)
\begin{proposition}\label{prop:invert}
Provided $|\wh{f}(\omega)|^2\neq 0$, the intensity measurements $\vd$
determine 
\begin{equation}\label{eq:recp}
\wt{\vp} \equiv \vp+(\ov{\vg}_0)^{-1} \odot \vg_0 \odot \ov{\vp},
\end{equation}
where the inverse of a
vector is understood componentwise. Moreover $\wt{\vp}$ can be obtained
in about $2N$ complex operations from $\vd$ with
\begin{equation}
 \label{eq:recpbis}
 \wt{\vp} = |\wh{f}(\omega)|^{-2} (\ov{\vg}_0)^{-1} \odot \vd -
\vg_0.
\end{equation}
\end{proposition}
\begin{proof}
Since we use the first (resp. last) $N$ rows of $\mM^\dagger$ to
recover the real (resp. imaginary) part of a vector in $\complex^N$, it
is convenient to consider the matrix
\[
  \begin{bmatrix} \mI & i \mI \end{bmatrix} \mM^\dagger = \diag(\vg_0)
  \diag(\ov{\vg}_0 \odot \vg_0)^{-1} = 
  \diag(\ov{\vg}_0)^{-1},
\]
where $\mI$ is the $N\times N$ identity matrix. To see what information
about $\vp$ we can recover from the right hand side $\vd$ in the least-squares system
\eqref{eq:linsys} we can evaluate:
\[
 \begin{aligned}
  |\wh{f}(\omega)|^{-2} \begin{bmatrix} \mI & i \mI \end{bmatrix}
  \mM^\dagger \vd &= 
  \diag(\ov{\vg}_0)^{-1} \mM \begin{bmatrix} \Re(\vg_0 + 2\vp)\\
 \Im(\vg_0 + 2\vp)\end{bmatrix}\\
 &=  \diag(\ov{\vg}_0)^{-1} [  \Re(\vg_0)^2 + \Im(\vg_0)^2  + 2 \Re(\vg_0)
 \Re(\vp) + 2\Im(\vg_0) \Im(\vp) ]\\
 &= \vg_0 + \diag(\ov{\vg}_0)^{-1}(\vg_0 \odot \ov{\vp} + \ov{\vg}_0
 \odot \vp)\\
 &=\vg_0 + \vp + \vg_0 \odot(\ov{\vg}_0)^{-1} \odot \ov{\vp} =\vg_0 +
 \wt{\vp}.
 \end{aligned}
\]
Hence we can get $\wt{\vp}$ from the intensity data $\vd$ with
essentially $N$ complex multiplications and $N$ complex additions.

\end{proof}
A natural question to ask is whether we can obtain $\wt{\vp}$ in a
stable manner from $\vd$. This can be answered by looking at the
conditioning of $\mM$, i.e. the ratio of the largest singular value
$\sigma_1$ of $\mM$ to $\sigma_N$, the smallest one. These are easily
obtained from the square roots of the eigenvalues of the diagonal matrix
$\mM \mM^T  = \diag(\ov{\vg}_0 \odot \vg_0)$. Hence the conditioning of
$\mM$ is the ratio of the largest to the smallest moduli of the entries
of $\vg_0$:
\begin{equation}\label{eq:condition}
\cond(\mM) = \begin{cases}
\displaystyle \frac{\max_r
\big|H_0^{(1)}(k|\vec{\vx}_r-\vec{\vx}_s|)\big|}{\min_r
\big|H_0^{(1)}(k|\vec{\vx}_r-\vec{\vx}_s|)\big|}, & \text{for $d=2,$}\\
\displaystyle\frac{\max_r\big|\vec{\vx}_r-\vec{\vx}_s\big|}{\min_r\big|\vec{\vx}_r-\vec{\vx}_s\big|},
&\text{for $d=3.$}
\end{cases}
\end{equation}

In figure~\ref{fig:condition} we show the condition number of
$\mM(\vec{\vx}_s,\omega)$ plotted over an optical frequency band. The
experimental setup is that given in \S\ref{sec:numerics}. The condition
number \eqref{eq:condition} is clearly independent of frequency for
$d=3$ and for $d=2$ we have the approximation for high frequencies: 
\[
\cond\mM(\vec{\vx}_s,\omega) =
\frac{\max_r|\vec{\vx}_r-\vec{\vx}_s|^{1/2}}{\min_r|\vec{\vx}_r-\vec{\vx}_s|^{1/2}}(1+\bigO(1/\omega)),\quad\text{as
$\omega\to\infty$.}
\]
This approximation follows from the Hankel function asymptotic (see e.g.
\cite{Olver:2010:NHMF})
\[
H_0^{(1)}(t) = \sqrt{\frac{2}{\pi
t}}\exp[\ii(t-\pi/4)](1+\bigO(1/t)),\quad\text{as $t\to\infty$.}
\] 
Thus the conditioning of $\mM$ is determined by the ratio of largest to
smallest source-to-receiver distances.
%%%%%%%%%%%%%%%%%%%%%%%%%%%%%%%%%%%%%%%%%%%%%%%%%%
%% FIGURE - CONDITION
\begin{figure}[!hbtp]
\centering
\includegraphics[width=0.5\textwidth]{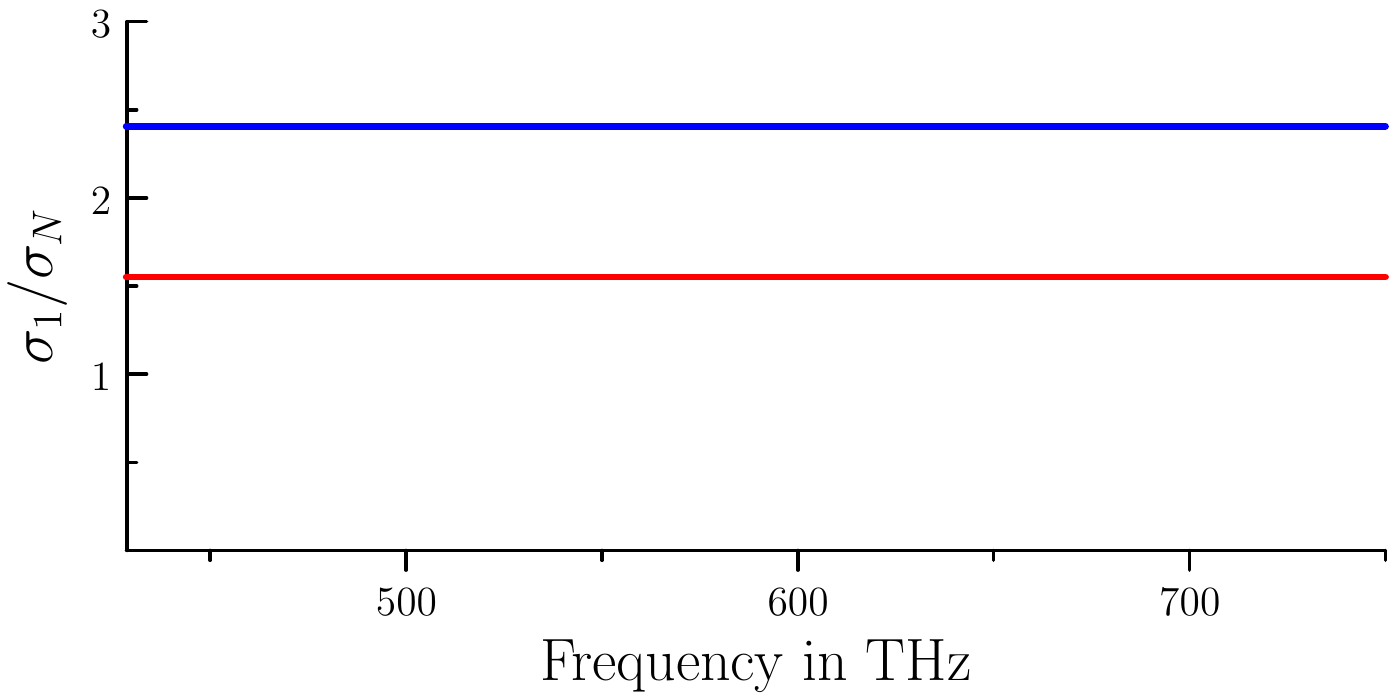}
\caption{Condition number of $\mM(\vec{\vx}_s,\omega)$ for $d=2$ (red)
and $d=3$ (blue) for the setup given in \S\ref{sec:numerics}}\label{fig:condition}
\end{figure}
%%%%%%%%%%%%%%%%%%%%%%%%%%%%%%%%%%%%%%%%%%%%%%%%%%

%%%%%%%%%%%%%%%%%%%%%%%%%%%%%%%%%%%%%%%%%%%%%%%%%%
%% Kirchhoff migration
\subsection{Kirchhoff migration}\label{sec:kirchhoff} We now show that
migrating the recovered data $\wt{\vp}$ \eqref{eq:recp} using $\KM$
gives essentially the same image as migrating the true data $\vp$. We
establish this result by means of a stationary phase argument but in
order to do this, we need the following assumption on the location of
the source $\vec{\vx}_s$.

%%%%%%%%%%%%%%%%%%%%%%%%%%%%%%
%% ASSUMPTION 2 (GEOMETRIC IMAGING)
\begin{assumption}[Geometric imaging conditions]\label{assump:geoimg}
For a scattering potential with support contained inside an image window
$\sW$, we assume $\vec{\vx}_s$ satisfies
\[
\frac{\vec{\vx}_r-\vec{\vx}_s}{|\vec{\vx}_r-\vec{\vx}_s|} \neq
\frac{\vec{\vx}_r-\vec{\vy}}{|\vec{\vx}_r-\vec{\vy}|},
\]
for $r=1,\ldots,N,$ and $\vec{\vy}\in\sW$.
\end{assumption}
%%%%%%%%%%%%%%%%%%%%%%%%%%%%%%

We interpret this assumption as a restriction on the placement of our
source location $\vec{\vx}_s$ as follows. Fix a receiver position $\vec{\vx}_r$ and
consider the cone 
\[
\sK(\vec{\vx}_r) =
\left\{\alpha\frac{\vec{\vy}-\vec{\vx}_r}{|\vec{\vy}-\vec{\vx}_r|}:\alpha>0,\vec{\vy}\in\sW\right\}.
\]
As long as $\vec{\vx}_s\notin\sK(\vec{\vx}_r)$, then we have that
$(\vec{\vx}_s-\vec{\vx}_r)/|\vec{\vx}_s-\vec{\vx}_r|\neq(\vec{\vy}-\vec{\vx}_r)/|\vec{\vy}-\vec{\vx}_r|$
for any $\vec{\vy}\in\sW$, i.e.\@ assumption~\ref{assump:geoimg} holds for
$\vec{\vx}_r$.  Ensuring this is satisfied for all receiver locations
$\vec{\vx}_r$ for $r=1,\ldots,N$, we require
$\vec{\vx}_s\notin\cup_{r=1}^N\sK(\vec{\vx}_r)$. In
figure~\ref{fig:geoimg} we illustrate this assumption. Here, the dark
blue region depicts the cone $\sK(\vec{\vx}_r)$ while the union of cones
$\cup_{r=1}^N\sK(\vec{\vx}_r)$ is depicted by the light blue region.
Assumption~\ref{assump:geoimg} simply requires $\vec{\vx}_s$ to be
outside the light blue region.

\begin{figure}[!hbtp]
\includegraphics[width=0.7\textwidth]{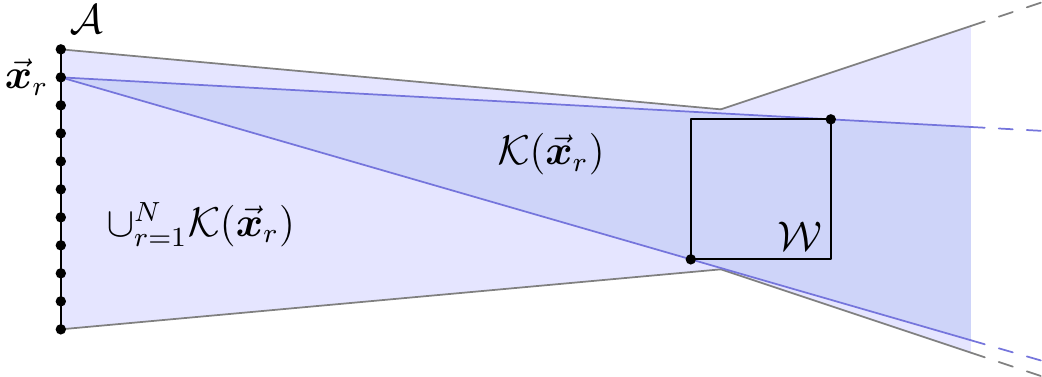}
\caption{Illustration of assumption~\ref{assump:geoimg}. If 
$\vec{\vx}_s$ is outside of the light blue region then
$(\vec{\vx}_r-\vec{\vx}_s)/|\vec{\vx}_r-\vec{\vx}_s|\neq
(\vec{\vx}_r-\vec{\vy})/|\vec{\vx}_r-\vec{\vy}|$ for all
$\vec{\vx}_r\in\sA$ and $\vec{\vy}\in\sW$.}\label{fig:geoimg}
\end{figure}

%%%%%%%%%%%%%%%%%%%%%%%%%%%%%%
%% THEOREM (IMAGING)
\begin{theorem}\label{thm:km}
Provided assumption~\ref{assump:geoimg} holds, the image of the
reconstructed array response vector is
\[
\KM\big[\vp+(\ov{\vg}_0)^{-1}\odot\vg_0\odot\ov{\vp},\omega\big](\vec{\vy})
\approx \KM\big[\vp,\omega\big](\vec{\vy}).
\]
\end{theorem}
%%%%%%%%%%%%%%%%%%%%%%%%%%%%%%

\begin{proof}
We begin by approximating the Kirchhoff migration functional
\eqref{eq:kirchhoff} by an integral over the array $\sA$:
\begin{equation}\label{eq:applykm2}
\begin{aligned}
\KM&\left[(\ov{\vg}_0)^{-1}\odot\vg_0\odot\ov{\vp}\right](\vec{\vy})\\
&=\ov{\wh{G}}_0(\vec{\vx}_s,\vec{\vy},\omega)\vg_0(\vec{\vy},\omega)^*\left[(\ov{\vg}_0(\vec{\vx}_s,\omega))^{-1}\odot\vg_0(\vec{\vx}_s,\omega)\odot\ov{\vp}(\vec{\vx}_s,\omega)\right]\\
&\sim k^2\int_\sR
d\vec{\vz}\int_{\sA}d{\vx}_rC(\vec{\vx}_s,\vec{\vx}_r,\vec{\vz},\vec{\vy})\times\\
&\quad\exp\left(\ii k(2|\vec{\vx}_r-\vec{\vx}_s|-|\vec{\vx}_r-\vec{\vz}|-|\vec{\vz}-\vec{\vx}_s|-|\vec{\vx}_r-\vec{\vy}|-|\vec{\vy}-\vec{\vx}_s|)\right).
\end{aligned}
\end{equation}
Here $\sim$ denotes equal up to a constant and
$C(\vec{\vx}_s,\vec{\vx}_r,\vec{\vz},\vec{\vy})$ is a smooth real valued
function that collects the various $\wh{G}_0$ geometric spreading terms.

Now we apply the method of stationary phase (see e.g.\@
\cite{Bleistein:2001:MMI}) to the integral over $\sA$. In the large
wavenumber limit $k\to\infty$, dominant contributions to the
integral come from stationary points of the phase, i.e.\@ points
$\vec{\vx}_r$ that satisfy
\[
\nabla_{\vec{\vx}_r}\left(2|\vec{\vx}_r-\vec{\vx}_s|-|\vec{\vx}_r-\vec{\vz}|-|\vec{\vz}-\vec{\vx}_s|-|\vec{\vx}_r-\vec{\vy}|-|\vec{\vy}-\vec{\vx}_s|\right)
= 0.
\]
This expression is equivalent to
\begin{equation}\label{eq:convexcomb}
\frac{\vec{\vx}_r-\vec{\vx}_s}{|\vec{\vx}_r-\vec{\vx}_s|}=
\frac{1}{2}\left(\frac{\vec{\vx}_r-\vec{\vy}}{|\vec{\vx}_r-\vec{\vy}|}+\frac{\vec{\vx}_r-\vec{\vz}}{|\vec{\vx}_r-\vec{\vz}|}\right).
\end{equation}
If \eqref{eq:convexcomb} holds then we must have
\begin{equation}\label{eq:cs}
\begin{aligned}
\left|\frac{\vec{\vx}_r-\vec{\vx}_s}{|\vec{\vx}_r-\vec{\vx}_s|}\right|^2 =
\frac{1}{4}\left|\frac{\vec{\vx}_r-\vec{\vy}}{|\vec{\vx}_r-\vec{\vy}|}+\frac{\vec{\vx}_r-\vec{\vz}}{|\vec{\vx}_r-\vec{\vz}|}\right|^2\\
\iff 1 =
\frac{\vec{\vx}_r-\vec{\vy}}{|\vec{\vx}_r-\vec{\vy}|}\cdot\frac{\vec{\vx}_r-\vec{\vz}}{|\vec{\vx}_r-\vec{\vz}|}.
\end{aligned}
\end{equation}
Since $(\vec{\vx}_r-\vec{\vy})/|\vec{\vx}_r-\vec{\vy}|$ and
$(\vec{\vx}_r-\vec{\vz})/|\vec{\vx}_r-\vec{\vz}|$ are both unit
vectors, it follows from the Cauchy-Schwarz \emph{equality} that
\eqref{eq:cs} holds only if $\vec{\vz}=\vec{\vy}$. Thus stationary
points must satisfy 
\[
\frac{\vec{\vx}_r-\vec{\vx}_s}{|\vec{\vx}_r-\vec{\vx}_s|} =
\frac{\vec{\vx}_r-\vec{\vy}}{|\vec{\vx}_r-\vec{\vy}|},
\]
where $\vec{\vy}\in\sW$. By assumption~\ref{assump:geoimg} there are no
such stationary points and therefore, neglecting boundary effects, this
integral vanishes faster than any polynomial power of $\omega$ (see
e.g.\@ \cite{Bleistein:1986:AEI}).
\end{proof}

%%%%%%%%%%%%%%%%%%%%%%%%%%%%%%%%%%%%%%%%%%%%%%%%%
%% Remark (applied to multiple sources single receiver)
\begin{remark}\label{rem:previous} We used a similar idea in
\cite{Bardsley:2015:PCSP} to show that with multiple sources, a single
receiver, and a specific pairwise illumination scheme it is possible to
image with sole knowledge of the intensities of the wave fields at the
receiver and of the probing fields. We approached the problem by
estimating the array response vector (a vector in $\complex^N$) with
$2N$ (or more) real measurements, which are essentially the measured
intensities for $2N$ or more different pairs of sources. The results of
\S \ref{sec:nullspace} and \S\ref{sec:kirchhoff} can be modified by
reciprocity to apply to the setup we considered in
\cite{Bardsley:2015:PCSP}. Hence images similar to those in
\cite{Bardsley:2015:PCSP} can be obtained without the pairwise
illumination scheme and the number of required illuminations is reduced
from $3N$ to $N$.  \end{remark}

%%%%%%%%%%%%%%%%%%%%%%%%%%%%%%%%%%%%%%%%%%%%%%%%%%%%%%%%%%%%%%%%%%%%%%%%
%%% Stochastic illuminations
\section{Stochastic illuminations and autocorrelations}\label{sec:stocillum} 
Our imaging method can also be used when the
source is driven by a stationary stochastic process (for which we only
assume knowledge of the autocorrelation or power spectra) and only
empirical autocorrelations are measured at the receiver locations. 

To be more precise, the source at $\vec{\vx}_s$ is driven by $f(t)$, a
stationary mean zero Gaussian process with autocorrelation function 
\begin{equation}\label{eq:fautocorr}
\langle \ov{f}(t)f(t+\tau)\rangle = F(\tau).
\end{equation}
Here $\langle\cdot\rangle$ denotes expectation with respect to
realizations of $f$ and we recall that $F(\tau) = \ov{F}(-\tau)$. In the
time domain, the field recorded at $\vec{\vx}_r$ is
\[
u(\vec{\vx}_r,\vec{\vx}_s,t) = \frac{1}{2\pi}\int d\omega e^{-\ii\omega t} 
\wh{G}(\vec{\vx}_r,\vec{\vx}_s,\omega)\wh{f}(\omega),\qquad\text{for
$r=1,\ldots,N$},
\]
where we assume $\wh{G}$ is given by the Born approximation
\eqref{eq:bornappx}.

The measurements at the receiver locations $\vec{\vx}_r$ are the
empirical autocorrelations:
\begin{equation}\label{eq:empautocorr}
\C(\vec{\vx}_r,\vec{\vx}_s,\tau) = \frac{1}{2T}\int_{-T}^Tdt
\ov{u}(\vec{\vx}_r,\vec{\vx}_s,t)u(\vec{\vx}_r,\vec{\vx}_s,t+\tau)\qquad\text{for
$r=1,\ldots,N,$}
\end{equation}
where $T$ is a fixed acquisition time. As shown by
\citet{Garnier:2009:PSI}, these measurements are independent of the
acquisition time $T$ and ergodic as we summarize in the following
proposition.

%%%%%%%%%%%%%%%%%%%%%%%%%%%%%%%%%%%%%%%%%%%%%%%%%%
% Ergodic
\begin{proposition}\label{prop:ergodic}
Assume $f(t)$ is a stationary mean zero Gaussian process satisfying
\eqref{eq:fautocorr}. The expectation of the empirical autocorrelations
\eqref{eq:empautocorr} is independent of the acquisition time $T$:
\[
\langle \C(\vec{\vx}_r,\vec{\vx}_s,\tau) \rangle =
\CE(\vec{\vx}_r,\vec{\vx}_s,\tau),
\]
where
\begin{equation}\label{eq:expautocorr}
\begin{aligned}
\CE(\vec{\vx}_r,\vec{\vx}_s,\tau) = \frac{1}{2\pi}\int d\omega e^{-\ii\omega \tau}\wh{F}(\omega)
\ve_r^\tr\left[\ov{\vg}(\vec{\vx}_s,\omega)\odot\vg(\vec{\vx}_s,\omega)\right],
\end{aligned}
\end{equation}
with $\vg\equiv \vg_0+\vp$.  Furthermore, \eqref{eq:empautocorr} is
ergodic, i.e.\@
\begin{equation}\label{eq:ergodic}
\C(\vec{\vx}_r,\vec{\vx}_s,\tau) \xrightarrow{T\to\infty}
\CE(\vec{\vx}_r,\vec{\vx}_s,\tau).
\end{equation}
\end{proposition}
%%%%%%%%%%%%%%%%%%%%%%%%%%%%%%%%%%%%%%%%%%%%%%%%%%
\begin{proof}
The proof is a straight-forward application of \cite[Proposition
4.1]{Garnier:2009:PSI}.  
\end{proof}

The ergodicity \eqref{eq:ergodic} of this proposition guarantees that for
sufficiently large acquisition time $T$ , the autocorrelation
$\C(\vec{\vx}_r,\vec{\vx}_s,\tau)$ is close to an intensity measurement,
i.e.
\[
\wh{\C}(\vec{\vx}_r,\vec{\vx}_s,\omega) \xrightarrow{T\to\infty}
\wh{\CE}(\vec{\vx}_r,\vec{\vx}_s,\omega)=\wh{F}(\omega)\ve_r^\tr\left[
\ov{(\vg_0+\vp)} \odot (\vg_0+\vp) \right].
\]
Proceeding analogously as in \S\ref{sec:ionly}, we neglect the quadratic
term in $\vp$:
\[
\wh{\CE}(\vec{\vx}_r,\vec{\vx}_s,\omega)
\approx\wh{F}(\omega)\ve_r^\tr\Re\left(\ov{\vg}_0 \odot (\vg_0+2\vp)\right).
\]
The collection of autocorrelations for $r=1,\ldots,N$ can be expressed,
approximately, as
\[
\left[\wh{\CE}(\vec{\vx}_r,\vec{\vx}_s,\omega)\right]_{r=1,\ldots,N}\\
\approx \vd(\vec{\vx}_s,\omega) \equiv
\wh{F}(\omega)\mM(\vec{\vx}_s,\omega)\begin{bmatrix}\Re(\vg_0+2\vp)\\\Im(\vg_0+2\vp)\end{bmatrix},
\]
where $\mM(\vec{\vx}_s,\omega)\in\real^{N\times 2N}$ is given by
\eqref{eq:measmx}. Therefore, the techniques developed in
\S\ref{sec:phaseret} can be applied to image from the autocorrelation
measurements \eqref{eq:empautocorr}.

%%%%%%%%%%%%%%%%%%%%%%%%%%%%%%%%%%%%%%%%%%%%%%%%%%%%%%%%%%%%%%%%%%%%%%%%
%% Numerics
\section{Numerical experiments}\label{sec:numerics} 
We now provide 2D numerical experiments of our proposed imaging method.
The physical scalings we use correspond to an optic regime. We use the
background wave velocity of $c_0 = 3\times 10^8$ m/s and central
frequency of about $590$ THz which gives a central wavelength
$\lambda_0$ of about $509$ nm. Our receiver array $\sA$ is a linear
array centered at the origin and consists of $501$ receivers located at
coordinates $\vec{\vx}_r=(0,-5+(r-1)(10/500))$mm for
$r=1,\ldots,501$. This corresponds to using a $1$ cm linear array of
receivers spaced approximately $20 \mu$m apart. We place the wave source
at coordinate $\vec{\vx}_s = (5,-7.5)$mm to guarantee
assumption~\ref{assump:geoimg} is satisfied. We begin with experiments
in the deterministic setting (\S\ref{sec:numdet}) followed by an
experiment in the stochastic setting (\S\ref{sec:numstoc}).  Lastly, we
investigate situations where assumptions~\ref{assump:smallness} and/or
\ref{assump:geoimg} are violated and our method is not expected to work
(\S\ref{sec:breakdown}). For all experiments, we assume 3D wave
propagation for simplicity so that $\wh{G}_0$ is given by
\eqref{eq:greenf} for $d=3$.

%%%%%%%%%%%%%%%%%%%%%%%%%%%%%%%%%%%%%%%%%%%%%%%%%%
\subsection{Deterministic illuminations}\label{sec:numdet} 
Using the illumination $\wh{f}(\omega)\equiv 1$, we generate the
intensity data $\vd(\vec{\vx}_s,\omega)$ using the Born approximation:
\[
\vd(\vec{\vx}_s,\omega) = \ov{(\vg_0+\vp)}\odot(\vg_0+\vp),
\]
with $\vg_0$ and $\vp$ defined by \eqref{eq:incvec} and
\eqref{eq:refvec} respectively, for $100$ uniformly spaced frequencies
in the frequency band $[430,750]$ THz. This corresponds to obtaining
intensity data for $100$ different monochromatic illuminations with
wavelengths $\lambda\in[400,700]$ nm, equally spaced in the frequency
band. Note that the quadratic term $\ov{\vp}\odot\vp$ is present in our data,
but using assumption~\ref{assump:smallness} we proceed assuming $\vd$ is
well approximated by the linear system \eqref{eq:linsys}. 

We recover the approximate array response vector
$\wt{\vp}=(\ov{\vg}_0)^{-1}\odot\vd-\vg_0$
for each frequency $\omega$ in the angular frequency band $\sB$, where
$(2\pi)^{-1}\sB = [430,750]$ THz. An image is then formed using the
Kirchhoff migration functional integrated over $\sB$: 
\[
\KM[\wt{\vp}](\vec{\vy}) = \int_{\sB} d\omega
\KM[\wt{\vp},\omega](\vec{\vy}),
\]
where $\KM$ is defined in \eqref{eq:kirchhoff}. Here we consider image
points $\vec{\vy}\in\sW =
\{(50\mbox{mm}+i\lambda_0/2.5,j\lambda_0/2.5),\text{ for
}i,j=-25,\ldots,25\}.$ 

For our first experiment, we consider a point reflector located at
coordinate $\vec{\vy}=(50,0)$mm with refractive index perturbation
$\rho(\vec{\vy})=1\times 10^{-15}$ (roughly equivalent to a reflector of area $(\lambda_0)^2$ and reflectivity $5978$). The migrated images of the true
array response vector $\vp$ and the recovered array response vector
$\wt{\vp}$ are shown in figure~\ref{fig:point}a. Although we are
significantly undersampling both in frequency and on the array (recall
the spacing between receivers is approx. $20 \mu$m $\gg\lambda_0/2$),
the images still exhibit the cross-range (Rayleigh) resolution estimate
$\lambda_0 L/a\approx5\lambda_0$ and range resolution estimate
$c_0/|\sB|\approx 1\lambda_0$. Our second experiment
(figure~\ref{fig:point}b) uses two point reflectors located at
coordinates $\vec{\vy}_1=(50\mbox{mm}-3\lambda_0,-\lambda_0)$ and
$\vec{\vy}_2=(50\mbox{mm}+6\lambda_0,5\lambda_0)$ each with
$\rho(\vec{\vy}_i)=1\times 10^{-15}$. We show an extended scatterer (a
disk) in figure~\ref{fig:extend}. The disk is generated as a set of
point reflectors, each with $\rho(\vec{\vy}_i)=1\times 10^{-15}$
separated by $\lambda_0/4$.

%%%%%%%%%%%%%%%%%%%%%%%%%%%%%%%%%%%%%%%%%%%%%%%%%%
%% FIGURE - CENTERED
\begin{figure}[!hbtp]
\centering
\begin{tabular}{p{0.5\textwidth}p{0.5\textwidth}}
\centering $\KM[\vp](\vec{\vy})$ &
\centering $\KM[\wt{\vp}](\vec{\vy})$
\end{tabular}
\includegraphics[width=\textwidth]{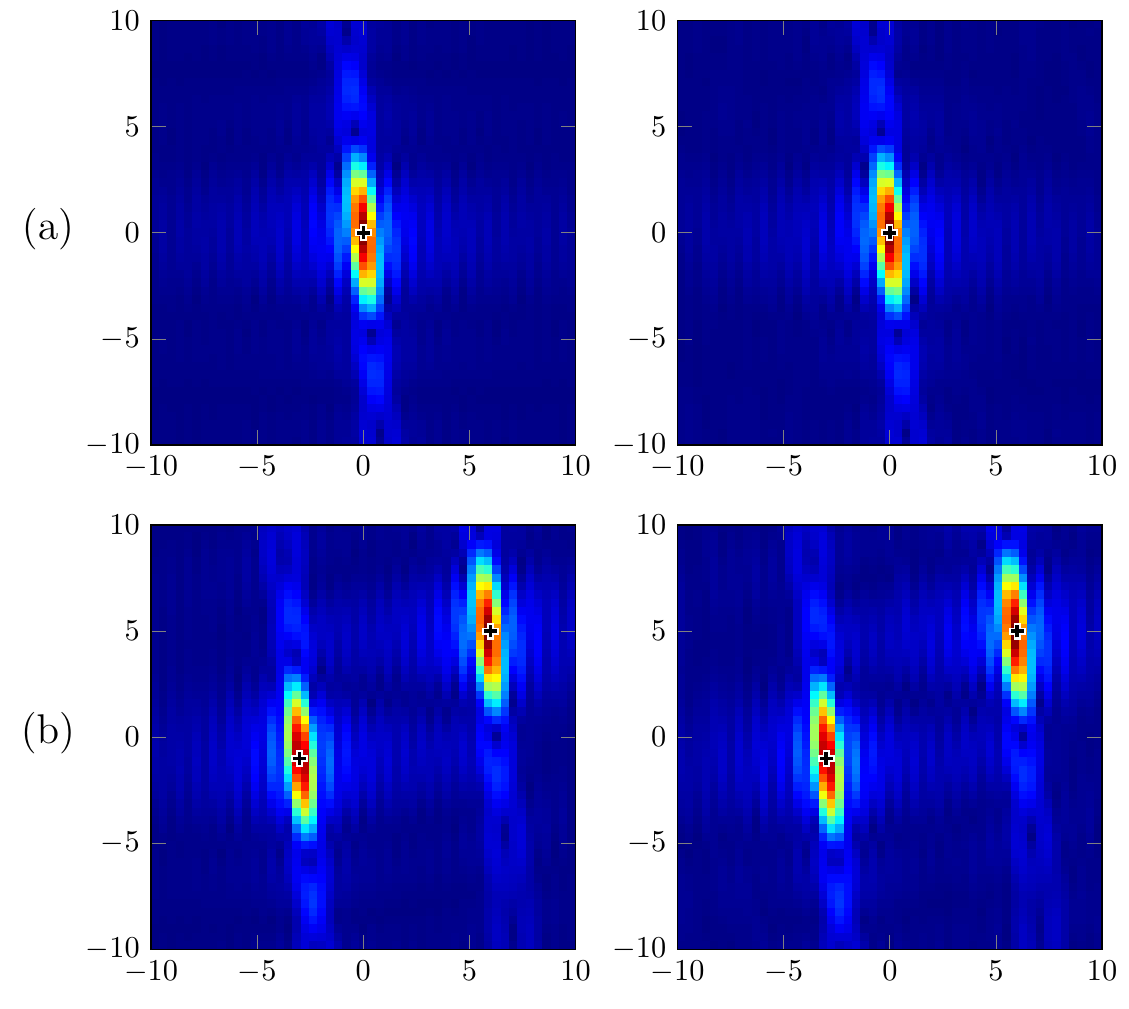}\\
\caption{Kirchhoff images of (a) one and (b) two point scatterers whose
true locations are indicated by crosses. The left column uses the full
array response vector $\vp$ while the right column uses the array
response vector $\wt{\vp}$ recovered from intensity data. The horizontal
and vertical axes display the range and cross-range respectively,
measured in central wavelengths $\lambda_0$ from
$(50,0)$mm.}\label{fig:point}
\end{figure}
%%%%%%%%%%%%%%%%%%%%%%%%%%%%%%%%%%%%%%%%%%%%%%%%%%

%%%%%%%%%%%%%%%%%%%%%%%%%%%%%%%%%%%%%%%%%%%%%%%%%%
%% FIGURE - EXTENDED 
\begin{figure}[!hbtp]
\centering
\begin{tabular}{p{0.5\textwidth}p{0.5\textwidth}}
\centering $\KM[\vp](\vec{\vy})$ &
\centering $\KM[\wt{\vp}](\vec{\vy})$
\end{tabular}
\includegraphics[width=\textwidth]{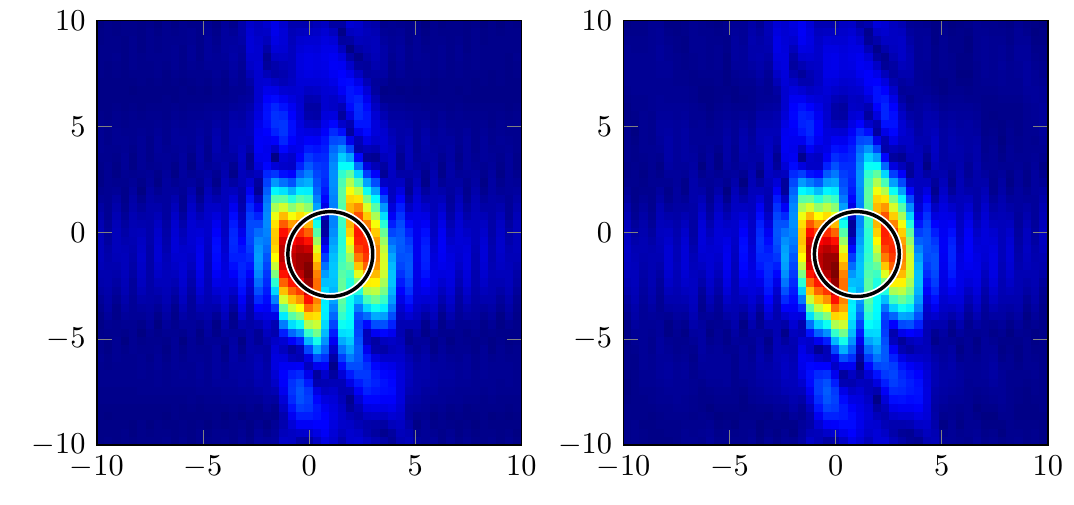} 
\caption{Kirchhoff images of an extended scatterer (disk). The boundary
of the disk is indicated by the black and white circle. The left image
uses the true array response vector $\vp$ while the right image uses the
array response vector $\wt{\vp}$ recovered from intensity measurements.
The horizontal and vertical axes display the range and cross-range
respectively, measured in central wavelengths $\lambda_0$ from
$(50,0)$mm.}\label{fig:extend}
\end{figure}
%%%%%%%%%%%%%%%%%%%%%%%%%%%%%%%%%%%%%%%%%%%%%%%%%%

%%%%%%%%%%%%%%%%%%%%%%%%%%%%%%%%%%%%%%%%%%%%%%%%%%
\subsection{Stochastic illumination}\label{sec:numstoc} 
Here we image with power spectrum data $\vd$ generated from a stochastic
illumination as in \S\ref{sec:stocillum}. Since we work in an optic
regime, adequately sampling signals in the time domain and performing
the autocorrelations \eqref{eq:empautocorr} is an expensive calculation.
We instead use the Wiener-Khinchin theorem \cite{Ishimaru:1997:WPS} to
simulate power spectrum measurements directly.

We assume the wave source at $\vec{\vx}_s$ is driven by a stationary
mean zero Gaussian process $f(t)$ with correlation function $\langle
\ov{f}(t)f(t+\tau)\rangle = F(\tau)$. By the Wiener-Khinchin theorem,
$\wh{f}(\omega)$ is a mean zero Gaussian process with correlation
function
\begin{equation}\label{eq:freqcorrfunc}
\langle \ov{\wh{f}}(\omega)\wh{f}(\omega')\rangle =
2\pi\delta(\omega-\omega')\wh{F}(\omega).
\end{equation}
Thus frequency samples of $\wh{f}(\omega)$ are \emph{independent} normal
random variables with variance proportional to $\wh{F}(\omega)$. Here we
use
\begin{equation}\label{eq:psd}
\wh{F}(\omega) =
t_c\exp\left(\frac{-(\omega-\omega_0)^2}{4\pi/t_c^2}\right),
\end{equation}
where $(2\pi)^{-1}\omega_0 = 590$ THz is the central frequency and $t_c
= 150\times 10^{-12}$ sec is the correlation time of $f(t)$ (i.e.
$F(\tau)\approx 0$ for
$\tau\gg t_c$). This choice of $t_c$ gives the signal an effective
frequency band of $(2\pi)^{-1}\sB = [430,750]$ THz (i.e.
$\wh{F}(\omega)\approx 0$ for $\omega\notin\sB$). Using
\eqref{eq:freqcorrfunc} and \eqref{eq:psd} we generate frequency
samples $\wh{f}(\omega_i)$ for $100$ frequencies $\omega_i$ equally
spaced in $\sB$.

For a large enough acqusition time $T$, the empirical autocorrelations
\eqref{eq:empautocorr} give frequency domain measurements proportional
to
\[
\wh{\C}(\vec{\vx}_r,\omega) =
\big|\ve_r^\tr(\vg_0+\vp)\wh{f}(\omega)\big|^2\quad\text{for
$r=1,\ldots,N$.}
\]
Thus for each frequency $\omega_i\in\sB$ we generate the power spectrum data 
\begin{equation}\label{eq:cleandata}
\vd(\vec{\vx}_s,\omega_i) =
\ov{\Big((\vg_0+\vp)\wh{f}(\omega_i)\Big)}\odot\Big((\vg_0+\vp)\wh{f}(\omega_i)\Big).
\end{equation}
Because correlations are robust with respect to additive noise, we also
consider autocorrelations with additive noise: 
\[
\wh{\C}(\vec{\vx}_r,\omega) =
\big|\ve_r^\tr(\vg_0+\vp)\wh{f}(\omega)+\wh{\eta}_r(\omega)\big|^2\quad\text{for
$r=1,\ldots,N$,}
\]
where the noise $\wh{\eta}_r(\omega)$ is an \emph{independent} mean zero
Gaussian process with correlation function given by
\eqref{eq:freqcorrfunc} and \eqref{eq:psd} for each $r=1,\ldots,N$. Here
we set the noise power equal to $10\%$ of the signal power at each
receiver, i.e.
\[
\int d\omega |\wh{\eta}_r(\omega)|^2 = \frac{1}{10}\int d\omega
|\ve_r^\tr(\vg_0+\vp)\wh{f}(\omega)|^2,\quad\text{for $r=1,\ldots,N$.}
\]
Noisy data for each frequency $\omega_i\in\sB$ is then generated as 
\begin{equation}\label{eq:noisydata}
\vd(\vec{\vx}_s,\omega_i) =
\ov{\Big((\vg_0+\vp)\wh{f}(\omega_i)+\wh{\veta}(\omega_i)\Big)}\odot\Big((\vg_0+\vp)\wh{f}(\omega_i)+\wh{\veta}(\omega_i)\Big),
\end{equation}
where $\wh{\veta}(\omega) = \begin{bmatrix}
\wh{\eta}_1(\omega),\cdots,\wh{\eta}_N(\omega)\end{bmatrix}^\tr.$ We can
indeed consider noise with much larger power (e.g. noise power equal to
$100\%$ signal power), however to compensate we then need additional
frequency samples to maintain sufficient averaging in migration images.
In figure~\ref{fig:stocillum} we show the migrated images $\KM[\wt{\vp}]$
for $\wt{\vp}=(2\pi\wh{F}\ov{\vg}_0)^{-1}\odot\vd-\vg_0$ recovered from clean
data \eqref{eq:cleandata} and from noisy data \eqref{eq:noisydata}.

%%%%%%%%%%%%%%%%%%%%%%%%%%%%%%%%%%%%%%%%%%%%%%%%%%
%% FIGURE - STOCHASTIC 
\begin{figure}[!hbtp]
\centering
\begin{tabular}{p{0.32\textwidth}p{0.32\textwidth}p{0.32\textwidth}}
\centering $\KM[\vp](\vec{\vy})$ &
\centering $\KM[\wt{\vp}](\vec{\vy})$ &
\centering $\KM[\wt{\vp}](\vec{\vy})$ 
\end{tabular}
\includegraphics[width=\textwidth]{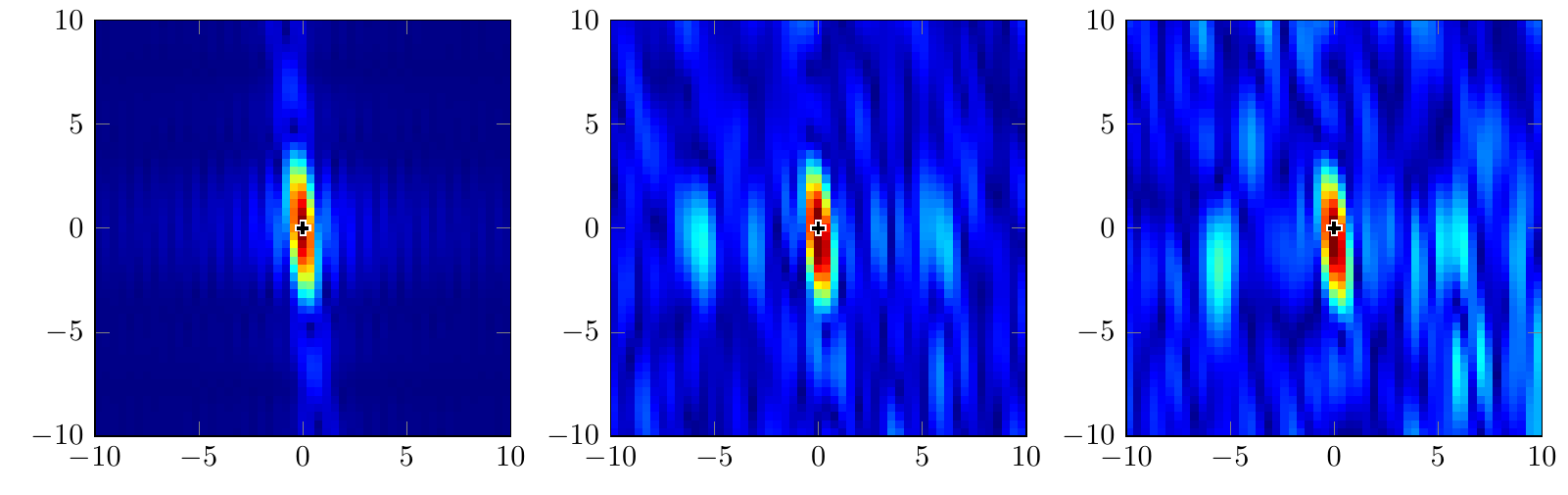} 
\caption{Kirchhoff images of a point scatterer using a stochastic
illumination and autocorrelation measurements. The images are generated
using (left) the true array response vector $\vp$, (center) $\wt{\vp}$
recovered from clean power spectrum data \eqref{eq:cleandata} and (right)
$\wt{\vp}$ recovered from noisy power spectrum data \eqref{eq:noisydata}.
The horizontal and vertical axes display the range and cross-range
respectively, measured in central wavelengths $\lambda_0$ from
$(50,0)$mm.}\label{fig:stocillum} 
\end{figure}
%%%%%%%%%%%%%%%%%%%%%%%%%%%%%%%%%%%%%%%%%%%%%%%%%%

%%%%%%%%%%%%%%%%%%%%%%%%%%%%%%%%%%%%%%%%%%%%%%%%%%
\subsection{Breakdown of the method}\label{sec:breakdown}
We now investigate situations where
assumptions~\ref{assump:smallness}~and/or~\ref{assump:geoimg} are
violated. For these experiments, we fix the receiver array $\sA$ (again
consisting of $501$ receivers with locations $\vec{\vx}_r$ given above)
while varying the source position $\vec{\vx}_s$, the reflector location
$\vec{\vy}$ and the reflectivity $\rho(\vec{\vy})$. In
figure~\ref{fig:breakdown} we show the migrated images of the recovered
array response vector $\KM[\wt{\vp}]$ for the following situations:\\
\begin{tabular}{l|l|c|c|c}
& \parbox{6em}{Assumptions\\violated\\[-0.9em]} & $\vec{\vy}$ & $\rho(\vec{\vy})$ &
$\vec{\vx}_s$\\\hline
(a) Source near scatterer & \ref{assump:smallness} and
\ref{assump:geoimg} & $(50\mbox{mm},0)$ & $10 ^{-15}$ & 
$(50\mbox{mm}-10\lambda_0,0)$\\
(b) Receivers near scatterer & \ref{assump:smallness} &
$(11\lambda_0,0)$ & $10^{-15}$ & $ (-50\mbox{mm},0)$\\
(c) Large reflectivity & \ref{assump:smallness} & $(50\mbox{mm},0)$ &
$10^{-10}$ & $(5,-75)$mm\\
(d) No geometric imaging condition & \ref{assump:geoimg} & $(50\mbox{mm},0)$ &
$10^{-15}$ & $(5\mbox{mm},0)$
\end{tabular}\\

From figures~\ref{fig:breakdown}a,~\ref{fig:breakdown}b and
\ref{fig:breakdown}c we see the imaging method is most sensitive to breaking assumption~\ref{assump:smallness}. In these situations the
quadratic term $\ov{\vp}\odot\vp$ cannot be neglected in the intensity data
\eqref{eq:intensity} and thus the linear system we consider in
\eqref{eq:linsys} is no longer a good approximation. This leads to artifacts in the images.
Figure~\ref{fig:breakdown}d demonstrates the imaging method is more
robust than expected with respect to assumption~\ref{assump:geoimg} and
the position of $\vec{\vx}_s$.

\begin{figure}[!hbtp]
\centering
\includegraphics[width=\textwidth]{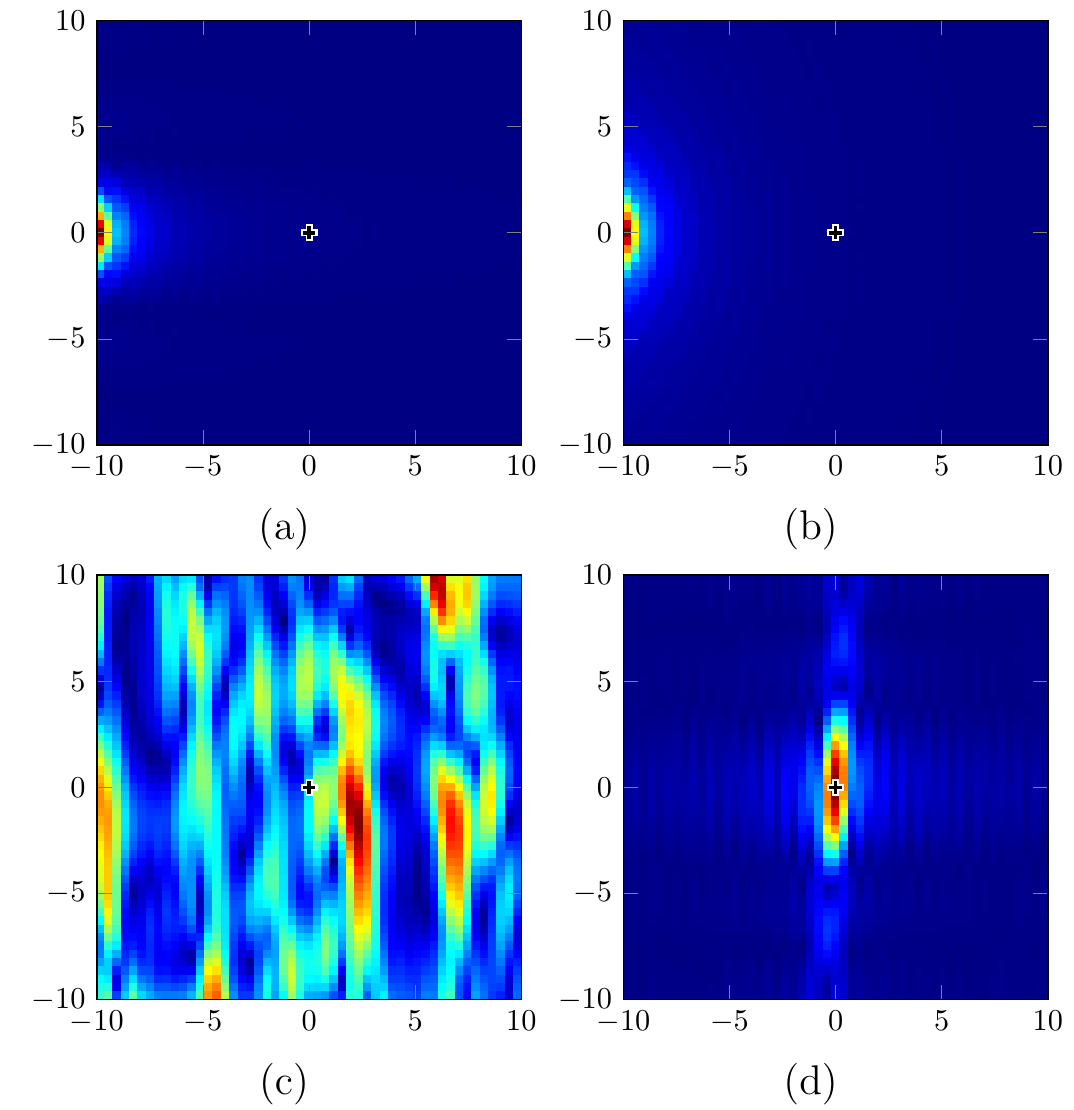}
\caption{Breakdown of imaging method: migrated images $\KM[\wt{\vp}]$
for setups violating assumptions \ref{assump:smallness} and/or
\ref{assump:geoimg}. Details of each setup are listed above and
correspond to (a) $\vec{\vx}_s$ placed to close to $\vec{\vy}$, (b)
$\sA$ placed too close to $\vec{\vy}$, (c) large reflectivity $\rho$ and
(d) $\vec{\vx}_s$ placed in front of the array $\sA$. The axes are
measured in central wavelengths $\lambda_0$ from the scatterer's true
location $\vec{\vy}$ which is indicated by a
cross.}\label{fig:breakdown}
\end{figure}

%%%%%%%%%%%%%%%%%%%%%%%%%%%%%%%%%%%%%%%%%%%%%%%%%%%%%%%%%%%%%%%%%%%%%%%%
%% Discussion
\section{Discussion and Future Work}\label{sec:discussion}

We have shown that when the scattered field is small compared to the
incident field (assumption~\ref{assump:smallness}), one can consider the
problem of recovering full-waveform data from intensity measurements as
a linear least-squares problem. For $N$ receivers, the corresponding
real matrix is $N \times 2N$ so all we can expect to recover is the
projection of the real and imaginary parts of the full-waveform data onto an $N$ dimensional subspace.
This turns out to be sufficient to image with Kirchhoff migration
(theorem~\ref{thm:km}). Crucially we do not need to manipulate the
fields at the receiver end, e.g. to introduce phases. The least-squares
problems we obtain are usually well-conditioned and the computational
cost of solving them ($\mathcal{O}(N)$ complex operations, each) is
negligible compared to the cost of Kirchhoff migration. Since we make no
assumptions on the source phases, our method adapts well to situations
where the source is driven by a Gaussian process and the measurements
are autocorrelations at the receiver locations.

The fundamental principle we have used here is that the imaging method
(in this case Kirchhoff migration) does not require all the data (in
this case the full-waveform scattered field) to form an image. For
Kirchhoff migration this is exploited e.g. by undersampling in frequency
and/or using only a few sources or receivers to image. We have shown
that there is another way in which one can use incomplete data, as
projections of the array response vector on certain subspaces leave the
Kirchhoff images unaffected. A similar principle is what is exploited by 
Novikov et al. \cite{Novikov:2014:ISI} to image with intensities, since
they show that knowing inner products of single source experiments is
enough to image with MUSIC. It would be interesting to carry this idea
further and see whether the same preprocessing we use here works for
MUSIC and also whether it is possible to image scatterers with even less
data. 

%%%%%%%%%%%%%%%%%%%%%%%%%%%%%%%%%%%%%%%%%%%%%%%%%%%%%%%%%%%%%%%%%%%%%%%%
\section*{Acknowledgements}
The work of P. Bardsley and F. Guevara Vasquez was partially supported
by the National Science Foundation grant DMS-1411577.

%%%%%%%%%%%%%%%%%%%%%%%%%%%%%%%%%%%%%%%%%%%%%%%%%%%%%%%%%%%%%%%%%%%%%%%%
%% BIBLIOGRAPHY
\bibliographystyle{abbrvnat}
\bibliography{biblspr}

\begin{thebibliography}{31}
\providecommand{\natexlab}[1]{#1}
\providecommand{\url}[1]{\texttt{#1}}
\expandafter\ifx\csname urlstyle\endcsname\relax
  \providecommand{\doi}[1]{doi: #1}\else
  \providecommand{\doi}{doi: \begingroup \urlstyle{rm}\Url}\fi

\bibitem[{Bardsley} and {Guevara Vasquez}(2015)]{Bardsley:2015:PCSP}
P.~{Bardsley} and F.~{Guevara Vasquez}.
\newblock {Imaging with power controlled source pairs}.
\newblock \emph{SIAM J. Imaging Sci.}, 2015.
\newblock Accepted for publication.

\bibitem[Bleistein and Handelsman(1986)]{Bleistein:1986:AEI}
N.~Bleistein and R.~A. Handelsman.
\newblock \emph{Asymptotic expansions of integrals}.
\newblock Dover Publications, Inc., New York, second edition, 1986.
\newblock ISBN 0-486-65082-0.

\bibitem[Bleistein et~al.(2001)Bleistein, Cohen, and
  Stockwell]{Bleistein:2001:MMI}
N.~Bleistein, J.~K. Cohen, and J.~W. Stockwell, Jr.
\newblock \emph{Mathematics of multidimensional seismic imaging, migration, and
  inversion}, volume~13 of \emph{Interdisciplinary Applied Mathematics}.
\newblock Springer-Verlag, New York, 2001.
\newblock ISBN 0-387-95061-3.
\newblock \doi{10.1007/978-1-4613-0001-4}.
\newblock Geophysics and Planetary Sciences.

\bibitem[Cand{\`e}s et~al.(2013)Cand{\`e}s, Strohmer, and
  Voroninski]{Candes:2013:PL}
E.~J. Cand{\`e}s, T.~Strohmer, and V.~Voroninski.
\newblock Phase{L}ift: exact and stable signal recovery from magnitude
  measurements via convex programming.
\newblock \emph{Comm. Pure Appl. Math.}, 66\penalty0 (8):\penalty0 1241--1274,
  2013.
\newblock ISSN 0010-3640.
\newblock \doi{10.1002/cpa.21432}.

\bibitem[Chai et~al.(2011)Chai, Moscoso, and Papanicolaou]{Chai:2011:AII}
A.~Chai, M.~Moscoso, and G.~Papanicolaou.
\newblock Array imaging using intensity-only measurements.
\newblock \emph{Inverse Problems}, 27\penalty0 (1):\penalty0 015005, 16, 2011.
\newblock ISSN 0266-5611.
\newblock \doi{10.1088/0266-5611/27/1/015005}.

\bibitem[Cheney(2001)]{Cheney:2001:LSM}
M.~Cheney.
\newblock The linear sampling method and the {MUSIC} algorithm.
\newblock \emph{Inverse Problems}, 17\penalty0 (4):\penalty0 591--595, 2001.
\newblock ISSN 0266-5611.
\newblock \doi{10.1088/0266-5611/17/4/301}.
\newblock Special issue to celebrate Pierre Sabatier's 65th birthday
  (Montpellier, 2000).

\bibitem[Cheney and Borden(2009)]{Cheney:2009:FRI}
M.~Cheney and B.~Borden.
\newblock \emph{Fundamentals of radar imaging}, volume~79 of \emph{CBMS-NSF
  Regional Conference Series in Applied Mathematics}.
\newblock Society for Industrial and Applied Mathematics (SIAM), Philadelphia,
  PA, 2009.
\newblock ISBN 978-0-898716-77-1.
\newblock \doi{10.1137/1.9780898719291}.

\bibitem[Crocco et~al.(2004)Crocco, D'Urso, and Isernia]{Crocco:2004:ISPMCC}
L.~Crocco, M.~D'Urso, and T.~Isernia.
\newblock Inverse scattering from phaseless measurements of the total field on
  a closed curve.
\newblock \emph{J. Opt. Soc. Am. A}, 21\penalty0 (4):\penalty0 622--631, Apr
  2004.
\newblock \doi{10.1364/JOSAA.21.000622}.

\bibitem[Devaney(1989)]{Devaney:1989:SDIM}
A.~J. Devaney.
\newblock Structure determination from intensity measurements in scattering
  experiments.
\newblock \emph{Phys. Rev. Lett.}, 62:\penalty0 2385--2388, May 1989.
\newblock \doi{10.1103/PhysRevLett.62.2385}.

\bibitem[Fienup(1982)]{Fienup:1982:PRA}
J.~R. Fienup.
\newblock Phase retrieval algorithms: a comparison.
\newblock \emph{Appl. Opt.}, 21\penalty0 (15):\penalty0 2758--2769, Aug 1982.
\newblock \doi{10.1364/AO.21.002758}.

\bibitem[Garnier and Papanicolaou(2009)]{Garnier:2009:PSI}
J.~Garnier and G.~Papanicolaou.
\newblock Passive sensor imaging using cross correlations of noisy signals in a
  scattering medium.
\newblock \emph{SIAM J. Imaging Sci.}, 2\penalty0 (2):\penalty0 396--437, 2009.
\newblock ISSN 1936-4954.
\newblock \doi{10.1137/080723454}.

\bibitem[Garnier and Papanicolaou(2010)]{Garnier:2010:RAI}
J.~Garnier and G.~Papanicolaou.
\newblock Resolution analysis for imaging with noise.
\newblock \emph{Inverse Problems}, 26\penalty0 (7):\penalty0 074001, 22, 2010.
\newblock ISSN 0266-5611.
\newblock \doi{10.1088/0266-5611/26/7/074001}.

\bibitem[Garnier et~al.(2015)Garnier, Papanicolaou, Semin, and
  Tsogka]{Garnier:2015:SNR}
J.~Garnier, G.~Papanicolaou, A.~Semin, and C.~Tsogka.
\newblock Signal to {N}oise {R}atio {A}nalysis in {V}irtual {S}ource {A}rray
  {I}maging.
\newblock \emph{SIAM J. Imaging Sci.}, 8\penalty0 (1):\penalty0 248--279, 2015.
\newblock ISSN 1936-4954.
\newblock \doi{10.1137/140968677}.

\bibitem[Gbur and Wolf(2002)]{Gbur:2002:DTW}
G.~Gbur and E.~Wolf.
\newblock Diffraction tomography without phase information.
\newblock \emph{Opt. Lett.}, 27\penalty0 (21):\penalty0 1890--1892, Nov 2002.
\newblock \doi{10.1364/OL.27.001890}.

\bibitem[Gbur and Wolf(2004)]{Gbur:2004:ICS}
G.~Gbur and E.~Wolf.
\newblock The information content of the scattered intensity in diffraction
  tomography.
\newblock \emph{Inform. Sci.}, 162\penalty0 (1):\penalty0 3--20, 2004.
\newblock ISSN 0020-0255.
\newblock \doi{10.1016/j.ins.2003.01.001}.

\bibitem[Gerchberg and Saxton(1972)]{Gerchberg:1972:GSA}
R.~Gerchberg and W.~Saxton.
\newblock A practical algorithm for the determination of phase from image and
  diffraction plane pictures.
\newblock \emph{Optik}, 35, 1972.

\bibitem[Holmes and Belen'kii(2004)]{Holmes:2004:CRE}
R.~B. Holmes and M.~S. Belen'kii.
\newblock Investigation of the {C}auchy--{R}iemann equations for
  one-dimensional image recovery in intensity interferometry.
\newblock \emph{J. Opt. Soc. Am. A}, 21\penalty0 (5):\penalty0 697--706, May
  2004.
\newblock \doi{10.1364/JOSAA.21.000697}.
\newblock URL \url{http://josaa.osa.org/abstract.cfm?URI=josaa-21-5-697}.

\bibitem[Ishimaru(1997)]{Ishimaru:1997:WPS}
A.~Ishimaru.
\newblock \emph{Wave propagation and scattering in random media}.
\newblock IEEE/OUP Series on Electromagnetic Wave Theory. IEEE Press, New York,
  1997.
\newblock ISBN 0-7803-3409-4.
\newblock Reprint of the 1978 original, With a foreword by Gary S. Brown, An
  IEEE/OUP Classic Reissue.

\bibitem[Maleki and Devaney(1993)]{Maleki:1993:PRI}
M.~H. Maleki and A.~J. Devaney.
\newblock Phase-retrieval and intensity-only reconstruction algorithms for
  optical diffraction tomography.
\newblock \emph{J. Opt. Soc. Am. A}, 10\penalty0 (5):\penalty0 1086--1092, May
  1993.
\newblock \doi{10.1364/JOSAA.10.001086}.

\bibitem[Novikov et~al.(2014)Novikov, Moscoso, and
  Papanicolaou]{Novikov:2014:ISI}
A.~Novikov, M.~Moscoso, and G.~Papanicolaou.
\newblock Illumination strategies for intensity-only imaging.
\newblock Preprint, 2014.
\newblock URL \url{http://arxiv.org/abs/1411.2655}.

\bibitem[Olver et~al.(2010)Olver, Lozier, Boisvert, and Clark]{Olver:2010:NHMF}
F.~W.~J. Olver, D.~W. Lozier, R.~F. Boisvert, and C.~W. Clark, editors.
\newblock \emph{{NIST Handbook of Mathematical Functions}}.
\newblock Cambridge University Press, New York, NY, 2010.

\bibitem[Schmitt(1999)]{Schmitt:1999:OCT}
J.~Schmitt.
\newblock Optical coherence tomography ({OCT}): a review.
\newblock \emph{Selected Topics in Quantum Electronics, IEEE Journal of},
  5\penalty0 (4):\penalty0 1205--1215, Jul 1999.
\newblock ISSN 1077-260X.
\newblock \doi{10.1109/2944.796348}.

\bibitem[Schmitt et~al.(1997)Schmitt, Lee, and Yung]{Schmitt:1997:OCM}
J.~Schmitt, S.~Lee, and K.~Yung.
\newblock An optical coherence microscope with enhanced resolving power in
  thick tissue.
\newblock \emph{Optics Communications}, 142\penalty0 (4–6):\penalty0 203 --
  207, 1997.
\newblock ISSN 0030-4018.
\newblock \doi{http://dx.doi.org/10.1016/S0030-4018(97)00280-0}.

\bibitem[Schuster(1996)]{Schuster:1996:RLC}
G.~T. Schuster.
\newblock Resolution limits for crosswell migration and traveltime tomography.
\newblock \emph{Geophysical Journal International}, 127\penalty0 (2):\penalty0
  427--440, 1996.
\newblock ISSN 1365-246X.
\newblock \doi{10.1111/j.1365-246X.1996.tb04731.x}.

\bibitem[Schuster(2009)]{Schuster:2009:SI}
G.~T. Schuster.
\newblock \emph{Seismic Interferometry}.
\newblock Cambridge University Press, 2009.
\newblock ISBN 978-0-521-87124-2.

\bibitem[Schuster et~al.(2004)Schuster, Yu, Sheng, and
  Rickett]{Schuster:2004:ISI}
G.~T. Schuster, J.~Yu, J.~Sheng, and J.~Rickett.
\newblock Interferometric/daylight seismic imaging.
\newblock \emph{Geophysical Journal International}, 157\penalty0 (2):\penalty0
  838--852, 2004.
\newblock \doi{10.1111/j.1365-246X.2004.02251.x}.

\bibitem[Takenaka et~al.(1997)Takenaka, Wall, Harada, and
  Tanaka]{Takenaka:1997:RAIM}
T.~Takenaka, D.~J.~N. Wall, H.~Harada, and M.~Tanaka.
\newblock Reconstruction algorithm of the refractive index of a cylindrical
  object from the intensity measurements of the total field.
\newblock \emph{Microwave and Optical Technology Letters}, 14\penalty0
  (3):\penalty0 182--188, 1997.
\newblock ISSN 1098-2760.
\newblock
  \doi{10.1002/(SICI)1098-2760(19970220)14:3<182::AID-MOP15>3.0.CO;2-A}.

\bibitem[Tarchi et~al.(2010)Tarchi, Lukin, Fortuny-Guasch, Mogyla, Vyplavin,
  and Sieber]{Tarchi:2010:SARNR}
D.~Tarchi, K.~Lukin, J.~Fortuny-Guasch, A.~Mogyla, P.~Vyplavin, and A.~Sieber.
\newblock {SAR} imaging with noise radar.
\newblock \emph{Aerospace and Electronic Systems, IEEE Transactions on},
  46\penalty0 (3):\penalty0 1214--1225, July 2010.
\newblock ISSN 0018-9251.
\newblock \doi{10.1109/TAES.2010.5545184}.

\bibitem[Teague(1983)]{Teague:1983:DPR}
M.~R. Teague.
\newblock Deterministic phase retrieval: a {G}reen's function solution.
\newblock \emph{J. Opt. Soc. Am.}, 73\penalty0 (11):\penalty0 1434--1441, Nov
  1983.
\newblock \doi{10.1364/JOSA.73.001434}.

\bibitem[Vela et~al.(2012)Vela, Narayanan, Gallagher, and
  Rangaswamy]{Vela:2012:NRT}
R.~Vela, R.~Narayanan, K.~Gallagher, and M.~Rangaswamy.
\newblock Noise radar tomography.
\newblock In \emph{Radar Conference (RADAR), 2012 IEEE}, pages 0720--0724, May
  2012.
\newblock \doi{10.1109/RADAR.2012.6212232}.

\bibitem[Yin and Xin(2015)]{Xin:2015:PLO}
P.~Yin and J.~Xin.
\newblock Phase{L}ift{O}ff: an accurate and stable phase retrieval method based
  on difference of trace and {F}robenius norms.
\newblock \emph{Commun. Math. Sci.}, 13\penalty0 (4):\penalty0 1033--1049,
  2015.
\newblock ISSN 1539-6746.
\newblock \doi{10.4310/CMS.2015.v13.n4.a10}.

\end{thebibliography}

\end{document}